\documentclass{article} %

\usepackage{lipsum}
\usepackage{graphicx}
\usepackage{algcompatible}
\usepackage{amsmath} %
\usepackage{amsthm}
\usepackage{amssymb}  %
\makeatletter
\renewenvironment{proof}[1][\proofname]{%
  \par\pushQED{\qed}%
  \normalfont\topsep6pt \trivlist
  \item[\hskip\labelsep\itshape
    \ifx#1\proofname
      Proof:
    \else
      Proof of #1:
    \fi
  ]\ignorespaces
}{%
  \popQED\endtrivlist\@endpefalse
}
\makeatother

\usepackage{fullpage}

\usepackage{amsfonts}
\usepackage{graphicx}
\usepackage{mathdots}
\usepackage{amsopn}

\usepackage{amssymb,amsfonts,amsmath,amsthm}
\usepackage{dsfont}
\let\mathbb=\mathds

\usepackage[english]{babel}
\usepackage{pgfplots}
\usepackage{tikz}
\usepackage{tkz-euclide}

\usepackage{enumitem}
\usepackage[normalem]{ulem}

\definecolor{ao(english)}{rgb}{0.0, 0.5, 0.0}
\usepackage[colorlinks, citecolor = {ao(english)}, linkcolor = {ao(english)}]{hyperref} 
\usepackage[nameinlink]{cleveref}

\usepackage{tikz}
\usepackage{pgfplots}
\usetikzlibrary{arrows,shapes,trees,calc,positioning,patterns,decorations.pathmorphing,decorations.markings}
\usetikzlibrary{matrix}
\usepgfplotslibrary{groupplots}
\pgfplotsset{compat=newest}

\newtheorem{thm}{Theorem}
\crefname{thm}{Theorem}{Theorems}
\newlist{thmenum}{enumerate}{1} %
\setlist[thmenum]{label=\roman*), ref=\thethm~\roman*)}
\crefalias{thmenumi}{thm} 
\newtheorem{prop}{Proposition}
\crefname{prop}{Proposition}{Propositions}
\newlist{propenum}{enumerate}{1} %
\setlist[propenum]{label=\roman*), ref=\theprop~\roman*)}
\crefalias{propenumi}{prop}

\newtheorem{lem}{Lemma}
\crefname{lem}{Lemma}{Lemmas}
\newlist{lemenum}{enumerate}{1} %
\setlist[lemenum]{label=\roman*), ref=\thelem~\roman*)}
\crefalias{lemenumi}{lem}

\newtheorem{cor}{Corollary}
\crefname{cor}{Corollary}{Corollaries}
\newlist{corenum}{enumerate}{1} %
\setlist[corenum]{label=\roman*), ref=\thecor~\roman*)}
\crefalias{corenumi}{cor} 

\newtheorem{rem}{Remark}
\crefname{rem}{Remark}{Remarks}
\newlist{remenum}{enumerate}{1} %
\setlist[remenum]{label=\roman*), ref=\therem~\roman*)}
\crefalias{remenumi}{rem} 

\newtheorem{example}{Example}
\crefname{example}{Example}{Examples}

\crefname{ass}{Assumption}{Assumption}

\crefname{conj}{Conjecture}{Conjectures}
\newtheorem{defn}{Definition}
\crefname{defn}{Definition}{Definitions}
\newlist{defnenum}{enumerate}{1} %
\setlist[defnenum]{label=\roman*., ref=\thedefn~(\roman*.)}
\crefalias{defnenumi}{defn}

\crefname{prob}{Problem}{Problems}

\newcommand{\rk}{\textnormal{rank}}

\newcommand{\sign}{\textnormal{sign}}

\newcommand{\transp}{\mathsf{T}}

\newcommand{\floor}[1]{\lfloor #1 \rfloor}
\newcommand{\ceil}[1]{\lceil #1 \rceil}

\newcommand{\Toep}[1]{\mathcal{T}^{#1}}

\newcommand{\Hank}[1]{\mathcal{H}^{#1}}

\newcommand{\permut}[1]{\textnormal{Perm}(#1)}

\colorlet{FigColor1}{blue}
\colorlet{FigColor2}{red}
\colorlet{FigColor3}{ao(english)}
\colorlet{FigColor4}{orange}
\pgfplotsset{every axis plot/.append style={line width=1.5pt}}
\crefformat{equation}{\textup{#2(#1)#3}}
\crefrangeformat{equation}{\textup{#3(#1)#4--#5(#2)#6}}
\crefmultiformat{equation}{\textup{#2(#1)#3}}{ and \textup{#2(#1)#3}}
{, \textup{#2(#1)#3}}{, and \textup{#2(#1)#3}}
\crefrangemultiformat{equation}{\textup{#3(#1)#4--#5(#2)#6}}%
{ and \textup{#3(#1)#4--#5(#2)#6}}{, \textup{#3(#1)#4--#5(#2)#6}}{, and \textup{#3(#1)#4--#5(#2)#6}}

\Crefformat{equation}{#2Equation~\textup{(#1)}#3}
\Crefrangeformat{equation}{Equations~\textup{#3(#1)#4--#5(#2)#6}}
\Crefmultiformat{equation}{Equations~\textup{#2(#1)#3}}{ and \textup{#2(#1)#3}}
{, \textup{#2(#1)#3}}{, and \textup{#2(#1)#3}}
\Crefrangemultiformat{equation}{Equations~\textup{#3(#1)#4--#5(#2)#6}}%
{ and \textup{#3(#1)#4--#5(#2)#6}}{, \textup{#3(#1)#4--#5(#2)#6}}{, and \textup{#3(#1)#4--#5(#2)#6}}

\crefdefaultlabelformat{#2\textup{#1}#3}

\usepackage[ backend=biber, style=numeric, eprint=true,doi=false,url=false,giveninits=true,style=numeric-comp]{biblatex}
\AtEveryBibitem{\clearfield{issn}}
\title{Efficient $k$-Sign Consistency Verification of Hankel Matrices via Schur Polynomials \thanks{This work was conducted while the first author was a Jane and Larry Sherman Fellow. It was further supported by the Israel Science Foundation (grant no.2406/22).}}

\author{Christian Grussler\thanks{C. Grussler is with the Stephen B. Klein Faculty of Aerospace Engineering, Technion --- Israel Institute of Technology, 3200003 Haifa, Israel
		{\tt cgrussler@technion.ac.il}} \; and \; Tobias Damm \thanks{T. Damm is with the Department of Mathematics, RPTU University Kaiserslautern-Landau, 67663 Kaiserslautern, Germany {\tt t.damm@rptu.de}}} 

\begin{document}

\maketitle

\begin{abstract}
    We consider the problem of certifying (strict) $k$-sign consistency of a matrix, that is, whether all of its $k$-th order minors share the same (strict) sign. Although this problem is generally of combinatorial complexity, we show that for Hankel matrices it can be significantly simplified: our sufficient condition requires checking only the $k$-th order minors of a reshaped Hankel matrix with $k$ rows. Remarkably, when applied to the Hankel operator, this sufficient condition is also necessary. Comparable results were known only in the setting of (strictly) $k$-positive Hankel matrices and operators, in which all minors of order up to $k$ have the same (strict) sign.

More concretely, we derive a formula expressing the $k$-th order minors of Hankel matrices as nonnegative integer linear combinations of $k$-th order minors with consecutive row indices. Our derivation uses Schur polynomial theory to show that the $k$-th order minors of any matrix are nonnegative integer linear combinations of row-consecutive $k$-th order minors, meaning minors formed from distinct columns whose consecutive row indices need not coincide across columns. For Hankel matrices, these minors coincide --- up to sign changes arising from column swaps --- with the usual $k$-th order minors with consecutive row indices. Our main result then follows by showing that the sum of certain signed nonnegative integer coefficients equals the corresponding Littlewood--Richardson coefficients. In our problem, the nonnegativity of these coefficients ensures that negatively signed column permutations are cancelled by positively signed ones. Our results also extend naturally to Toeplitz matrices and operators, and we present a partial analogue for circulant matrices.
\end{abstract}

\section{Introduction}

Total positivity theory of linear operators has attracted researchers from various fields over the past centuries. Positivity in the context of nonnegative matrices, i.e., matrices that map the cone of nonnegative vectors to itself, appears in probability theory (Markov chains), economics and biology (compartmental models), statistics (stochastic processes), graph theory (adjacency matrices), etc., where the restriction to nonnegative entries reflects the natural impossibility of negative probabilities, populations, flows, and so on (see, e.g., \cite{berman1979nonnegative,berman1989nonnegative,pinkus2009totally,karlin1968total}).

The framework of total positivity, however, also encompasses mappings with the more general invariance property of variation diminishing/bounding, i.e., reducing or preserving the number of sign changes (which equals zero for nonnegative vectors). This generalization of nonnegative matrix theory underlies, implicitly and explicitly, important developments as early as Descartes' Rule of Signs (1637) \cite{descartes1886geometrie}, the related Routh--Hurwitz criterion \cite{holtz2003hermite}, vibration mechanics \cite{gantmacher1950oszillationsmatrizen}, spline theory \cite{karlin1968total,gasca2013total}, geometric modelling \cite{gasca2013total,hagen1991variational}, and the omnipresent concepts of convexity and monotonicity \cite{grussler2025discrete,karlin1968total}. In recent years, this broader viewpoint has received increasing interest in the study of system behaviors (see, e.g., \cite{margaliot2018revisiting,grussler2020variation,weiss2019generalization2,weller2020strongly,grussler2025discrete,grussler2024system,tong2025selfsustained}), as well as in low-complexity modelling \cite{marmary2025tractabledownfallbasispursuit,grussler2020balanced}.

The term ``total positivity'' stems from the characterization of such linear mappings via the sign consistency of their minors \cite{karlin1968total,grussler2024system}. A matrix is called (i) \emph{(strictly) $k$-sign consistent} if all its $k$-th order minors share the same (strict) sign, (ii) \emph{(strictly) $k$-positive} if all its minors of order up to $k$ are (strictly) positive, and (iii) (strictly) totally positive if it is (strictly) $k$-positive for all $k$. As the number of minors grows combinatorially with the matrix dimensions, a major obstacle to the utility of this framework lies in efficient certification. Fortunately, in the case of (strict) $k$-positivity, it is under mild assumptions sufficient to check the (strict) positivity of consecutive minors of order up to $k$, i.e., minors formed from consecutive column and row indices \cite{karlin1968total,grussler2020variation,grussler2024system}. This not only significantly reduces the complexity of verifying the property, but when applied to Hankel matrices and operators also yields an exact and tractable characterization \cite{fallat2017total,grussler2020variation,pinkus2009totally}.

In this work, we show that similar simplifications can be obtained for (strictly) $k$-sign consistent Hankel matrices and operators. Our main contribution is an explicit formula that expresses each $k$-th order minor of a Hankel matrix as a nonnegative integer linear combination of $k$-th order minors from a reshaped Hankel matrix with $k$ rows. As a consequence, it is sufficient to verify (strict) sign consistency only among these $k$-th order minors to conclude (strict) $k$-sign consistency of the original Hankel matrix. In the case of the infinite-dimensional Hankel operator, where all $k$-th order minors of the reshaped Hankel matrix are also $k$-th order minors of the operator, this provides a necessary and sufficient characterization. Our result, in particular, facilitates the use of \cite{pena_matrices_1995}, which characterizes $k$-sign consistency of a matrix with $k$ rows via the total positivity of an associated transformed matrix. Finally, since column-order inversion of a Hankel matrix yields a Toeplitz matrix, our results also naturally extend to the Toeplitz case and partially to circulant matrices.

In our derivation, we use tools from algebraic combinatorics \cite{stanley2023enumerative} such as Schur polynomials, Kostka numbers, and Littlewood--Richardson coefficients. Specifically, we first utilize Schur polynomials to show that $k$-th order minors of a matrix can be expressed as nonnegative combinations of row-consecutive $k$-th order minors, i.e., minors obtained by selecting $k$ consecutive rows in each column, without requiring the row indices to coincide across columns. For Hankel matrices, the structure forces all such row-consecutive minors to coincide (up to sign) with the standard $k$-th order minors formed from the first $k$ rows of an extended/reshaped Hankel matrix. We then employ the definition of Littlewood--Richardson coefficients to show that all signed contributions arising from column permutations cancel, leaving only nonnegative coefficients in front of the consecutive-row minors. This yields the desired formula and the resulting efficient certificates for $k$-sign consistency.

The remainder of the paper is organized as follows: after extensive preliminaries in \Cref{sec:prelim} on total positivity and enumerative combinatorics, we derive and illustrate our main results in \Cref{sec:results}, before drawing conclusions in \Cref{sec:conculsion}.

\section{Preliminaries} \label{sec:prelim}
\subsection{Set Notations \& Sequences}
We write $\mathds{Z}$ for the set of integers and $\mathds{R}$ for the set of reals, with  $\mathds{Z}_{\ge 0}$ and $\mathds{R}_{\ge 0}$ standing for the respective subsets of nonnegative elements; the corresponding notation with strict inequality is also used for positive elements. For $k,l \in \mathds{Z}$, we use $(k:l) = \{k,k+1,\dots,l\}$ if $k \leq l$. For $r \in \mathds{Z}_{>0}$, we define the following sets of $r$-tuples: 
\begin{enumerate}
    \item the set of all $r$-tuples on $(k:l)$ is given by $(k:l)^r$.
    \item the set of increasing $r$-tuples on $(1:n)$ by 
\begin{equation*}
	\mathcal{I}_{n,r} := \{ v \in (1:n)^r: \; v_1 < v_2 < \dots  < v_r \},
\end{equation*}
\item the set of weakly decreasing $r$-tuples on $(0:n-1)$ by  
\begin{equation*}
    \mathcal{D}_{n,r} := \{ v \in (0:n-1)^r: \; v_1 \geq v_2 \geq \dots \geq v_r \}.
\end{equation*}
\item the set of weakly decreasing $r$-tuples associated with $\lambda \in  \mathcal{D}_{n,r}$ is given by 
\begin{equation*}
    \mathcal{D}_{\lambda} := \{ v \in  \mathcal{D}_{|\lambda|,r}: \sum_{i=1}^n (\lambda_i- v_i) = 0 \},
\end{equation*}
where $|\lambda| := \sum_{i=1}^n \lambda_i$.
\end{enumerate}
The $r$-tuple of all ones is denoted by $\mathbf{1}_r$. For two $r$-tuples $\lambda$ and $\mu$, we define their addition by $\lambda+\mu = (\lambda_1 + \mu_1,\dots,\lambda_r + \mu_r)$ 
and analogously the difference $\lambda-\mu$. We say that $\lambda \geq 0$ if all elements in $\lambda$ are nonnegative. Further, we use $\lambda^\downarrow$ and $\lambda^{\uparrow}$ to denote the $r$-tuple that results from sorting $\lambda$ in non-increasing order and non-decreasing order, respectively, and define $\max(\lambda) := (\lambda^{\downarrow})_1$. For $\lambda, \mu \in \mathcal{D}_{n,r}$, $\lambda \trianglerighteq \mu$ stands for $\mu$ preceding $\lambda$ in dominance order, i.e., $\sum_{i=1}^k \lambda_i \geq \sum_{i=1}^k \mu_i$ for all $k \in (1:r)$. %

The set of all sequences with indices in $\mathds{Z}$ and values in $\mathds{C}$ are defined by $\mathds{C}^\mathds{Z}$ and, similarly, we define $\mathds{C}^{\mathds{Z}_{\geq 0}}$, $\mathds{C}^{\mathds{Z}_{> 0}}$ and $\mathds{C}^{\mathds{Z}_{< 0}}$. $x \in \mathds{C}^\mathds{Z}$ is called $T$-periodic, if $x_{t} = x_{t+T}$ for all $t \in \mathds{Z}$ and the set of all $T$-periodic sequences is denoted by $\ell_\infty(T)$.  

\subsection{Matrices}
For matrices $X = (x_{ij}) \in \mathds{C}^{m \times n}$, the submatrix with rows $\mathcal{I} \subset (1:m)$ and columns $\mathcal{J} \subset (1:n)$ is written as $X_{(\mathcal{I},\mathcal{J})}$, where we also use the notions
$X_{(:,\mathcal{J})} := X_{((1:m),\mathcal{J})}$, $X_{(\mathcal{I},:)} := X_{(\mathcal{I},(1:n))}$. In the case of subvectors, we simply write $x_\mathcal{I}$. The determinant of $X \in \mathds{C}^{n \times n}$ is denoted by $|X|$. A \emph{(consecutive) $j$-minor} of $X \in \mathds{C}^{n \times m}$ is a minor that is constructed of (consecutive) $j$ columns and $j$ rows of $X$, i.e., $|X_{p,q}|$ with $p \in \mathcal{I}_{m,r}$, $q \in \mathcal{I}_{n,r}$ (where $p = (p_1:p_r)$ and $q = (q_1:q_r)$). The set of $j$-consecutive minors with column/row indices $(1:r)$ are said to be the $j$-\emph{column}/\emph{row} \emph{initial minors}. Further, for $p \in (1:n-r)^r$ and $q \in \mathcal{I}_{n,r}$ we say that
\begin{equation*}
    \begin{vmatrix}
            x_{p_{1}+1,q_1} & \dots & x_{p_{r}+1,q_r}\\
            \vdots &  & \vdots\\
            x_{p_{1}+r,q_1} & \dots & x_{p_r+r,q_r}
        \end{vmatrix}
\end{equation*}
is a \emph{row-consecutive $r$-minor} for $X$.  

\subsubsection{Structured Matrices}
The \emph{transposed Vandermonde matrix} $V_M(x) \in \mathds{C}^{M \times n}$ of length $M$ and in variables $x \in \mathds{C}^n$ is defined by
\begin{equation*}
   V_M(x) := \begin{bmatrix}
       1 & 1 & \dots & 1 \\
       x_1 & x_2 & \dots & x_n \\
       \vdots & \vdots & & \vdots \\
       x_1^{M-1} & x_2^{M-1} & \dots & x_n^{M-1} 
   \end{bmatrix}\,.
\end{equation*}
For $g \in \mathds{C}^{\mathds{Z}}$ and $t\in \mathds{Z}$ and $m,n \in \mathds{Z}_{\geq0}$ we define the \emph{Hankel matrices}
    \begin{align}
	H^g(t,m,n) &:= \begin{bmatrix}
	g_t & g_{t+1} & \dots & g_{t+n-1}\\
	g_{t+1} & g_{t+2} & \dots & g_{t+n}\\
	\vdots & \vdots &    & \vdots \\
	g_{t+m-1}   & g_{t+j} & \dots & g_{t+m+n-2}\\
	\end{bmatrix}
	\end{align}
and the \emph{Hankel operator} $\mathcal{H}^g: \mathds{C}^{\mathds{Z}_{< 0}} \to \mathds{C}^{\mathds{Z}_{\geq 0}}$ by
	\begin{align}
	 \label{eq:def_hank_disc}
	\Hank{g} u &:= \begin{bmatrix}
	    g_1 & g_2 & g_3 & \dots \\
        g_2 & g_3 & g_4 & \dots \\
        g_3 & g_4 & g_5 & \dots \\
        \vdots & \vdots & \vdots & \ddots 
	\end{bmatrix} \begin{bmatrix}
	    u_0\\
        u_{-1}\\
        u_{-2}\\
        \vdots 
	\end{bmatrix} \,.%
	\end{align}
Similarly, for $t \in \mathds{Z}$ and $m,n \in \mathds{Z}_{\geq 0}$, we define the Toeplitz matrices
\begin{equation}
    T^g(t,m,n) := \begin{bmatrix}
	g_{t} & g_{t-1} & \dots & g_{t-n+1}\\
	g_{t+1} & g_{t} & \dots & g_{t-n+2}\\
	\vdots & \vdots &  & \vdots \\
	g_{t+m-1}  & g_{t+m-2} & \dots & g_{t+m-n}\\
	\end{bmatrix}
\end{equation}
and the \emph{Toeplitz operator} $\mathcal{T}^g: \mathds{C}^{\mathds{Z}_{\geq 0}} \to \mathds{C}^{\mathds{Z}_{\geq 0}}$ by 
\begin{align}
	 \label{eq:def_toep_disc}
	\Toep{g} u &:= \begin{bmatrix}
	    g_{0} & g_{-1} & g_{-2} & \dots \\
        g_1 & g_0 & g_{-1} & \dots \\
        g_2 & g_1 & g_0 & \dots \\
        \vdots & \vdots & \vdots & \ddots 
	\end{bmatrix} \begin{bmatrix}
	    u_0\\
        u_{1}\\
        u_{2}\\
        \vdots 
	\end{bmatrix}\,. %
	\end{align}
Finally, for $g \in \ell_\infty(T)$, we define the \emph{circulant matrix}
\begin{equation}
    C^g = \begin{bmatrix}
        g_0 & g_{T-1} & \dots & g_1\\
        g_1 & g_{0} & \dots & g_2\\
        \vdots & \vdots  & & \vdots\\
        g_{T-1} & g_{T-2} & \dots & g_0
    \end{bmatrix} = T^g(0,T,T)\,.
\end{equation}
\subsubsection{Total Positivity}
The framework of total positivity characterizes linear mappings with the so-called variation diminishing and bounding properties, which have been utilized in interpolation theory \cite{Schoenberg1951polya,karlin1968total}, systems theory \cite{gantmacher1950oszillationsmatrizen,margaliot2018revisiting,weiss2019generalization2,grussler2020variation,grussler2020balanced,grussler2024system}, computer vision \cite{lindeberg1990scale}, sparse optimization \cite{marmary2025tractabledownfallbasispursuit}, etc. At the heart of this framework lie the following matrix notions.

\begin{defn}
Let $X \in \mathds{C}^{m \times n}$, and $k \in (1: \min \{m,n\})$. Then, $X$ is called 
\begin{itemize}
    \item \emph{(strictly) $k$-sign consistent} if all $k$-minors of $X$ share the same (strict) sign. 
    \item \emph{(strictly) $k$-positive} if all $j$-minors of $X$ are nonnegative (positive) for all $j \in (1:k)$. In case of $k = \min\{m,n\}$, $X$ is also called \emph{(strictly) totally positive}.
\end{itemize}

\end{defn}
While checking (strict) total positivity is generally of combinatorial complexity, under some strictness assumptions, it can be verified by only checking a polynomial amount of $j$-minors (see, e.g., \cite[Proposition~22 \& 23]{grussler2024system}).
\begin{prop}\label{prop:row_col_initial_minors_strict_tot_pos}
     Let $X \in \mathds{C}^{m \times n}$. Then the following hold:
     \begin{enumerate}
         \item $X$ is strictly totally positive if and only if all row and column initial minors of $X$ are positive.
         \item If $X$ is such that all $j$-row and $j$-column initial minors of $X$ are positive for all $j \in (1:\min\{n-1,m-1\})$ and nonnegative for $j = \min\{n,m\}$, then $X$ is totally positive and all minors of order less than $\min\{n,m\}$ are positive. 
     \end{enumerate}  
\end{prop}
Using repeated Dodgson condensation as in \cite{carter2021complexity}, it can be shown that the total complexity of \cref{prop:row_col_initial_minors_strict_tot_pos} sums up to $\mathcal{O}(m^2n)$ if $m \leq n$. Analogously to the proof of \cite[Theorem~2.3]{fallat2017total}, \cref{prop:row_col_initial_minors_strict_tot_pos} in conjunction with \cite[Proposition~8]{grussler2020variation} leads to the following (strict) $k$-positivity certificate. 
\begin{prop}
    \label{prop:consecutive_old}
	Let $X \in \mathds{C}^{m \times n}$, $k \leq \min \{m,n\}$, be such that
	\begin{enumerate}
	    \item all $j$-row and $j$-column initial minors of $X$ are positive for $j \in (1:k-1)$
	    \item all consecutive $k$-minors of $X$ are
	    nonnegative (positive).
	\end{enumerate}
	Then, $X$ is (strictly) $k$-positive. In the strict case, these are also necessary conditions.  
\end{prop}
In the special case of Hankel matrices, one can use \cref{prop:consecutive_old} to arrive at the following $k$-positivity certificate (see, e.g., \cite[Theorem~3.2]{fallat2017total} or \cite[Theorem~4]{grussler2020variation}).
\begin{cor}\label{lem:k_pos_Hankel}
    For $g \in \mathds{C}^{\mathds{Z}}$, $1 \leq k \leq \min\{M,N\}$ and $t \in \mathds{Z}$, the following are equivalent:
\begin{enumerate}
    \item $H^g(t,M,N)$ is (strictly) $k$-positive. 
    \item All consecutive $j$-minors of $H^g(t,M,N)$, $j \in (1:k)$, are nonnegative (positive). 
\end{enumerate}
\end{cor}
\cref{prop:consecutive_old} can also be used to verify $n$-sign consistency via the following bijection from \cite{pena_matrices_1995}.
\begin{lem}\label{lem:pena_bijection}
    Let $X \in \mathds{C}^{m \times n}, \; m >n$, be such that $|X_{((1:n),(1:n))}| \ne 0$ and 
    \begin{equation*}
       C :=  X_{((m-n+1:m),(1:n))} \left(X_{((1:n),(1:n))}\right)^{-1} K_n, 
    \end{equation*}
    where $K_n = (k_{ij}) \in \mathds{R}^{n \times n}$ is defined by
    \begin{equation*}
        k_{ij} = 
        \begin{cases}
            (-1)^{j-1} & i+j = n+1\\
            0 & \text{otherwise}
        \end{cases}.
    \end{equation*}    
    Then, $X$ is (strictly) $n$-sign consistent if and only if $C$ is (strictly) $n$-positive. 
\end{lem}
These results can be extended to the corresponding infinite-dimensional linear operators (see, e.g., \cite{grussler2020variation,karlin1968total,grussler2021internally}), which, in the case of the Hankel operator, leads to the following notions.
\begin{defn}
  Let $k \in \mathds{Z}_{>0}$, then $\Hank{g}$ is called 
  \begin{itemize}
      \item \emph{(strictly) $k$-sign consistent} if $H^g(1,N,N)$ is $k$-sign consistent for all $N \geq k$.
      \item \emph{(strictly) $k$-positive} if $H^g(1,N,N)$ is (stictly) $k$-positive for all $N \geq k$.
  \end{itemize}
 Analogous definitions are used for $\Toep{g}$ by replacing $H^g(1,N,N)$ with $T^g(0,N,N)$. 
\end{defn}
By  \cref{lem:k_pos_Hankel}, we can characterize the $k$-positivity of $\mathcal{H}^g$.
\begin{cor}\label{cor:k_pos_Hankel}
    Let $k \in \mathds{Z}_{>0}$, and $g \in \mathds{C}^\mathds{Z}$. Then, the following are equivalent:
    \begin{enumerate}
        \item $\mathcal{H}^g$ is (strictly) $k$-positive.
        \item $H^g(1,j,N)$ is (strictly) $j$-positive for all $N \geq j$ and all $j \in (1:k)$
        \item All consecutive $j$-minors of $H^g(t,j,N)$ are nonnegative (positive) for all $j \in (1:k)$ and all $N \geq j$.
    \end{enumerate}
\end{cor}

\subsection{Enumerative Combinatorics}
For an $r$-tuple $v = (v_i) \in (1:n)^r$, we use $\permut{v}$ to denote the set of all \emph{distinct} permutations. If $v$ has distinct entries, then the \emph{parity} of a permutation $p \in \permut{v}$ is defined by $\sign_v(p) = (-1)^m$, where $m$ is any number of transpositions needed to arrive from $v$ to $p$. Further, if $\sigma \in \permut{\{1,\dots,r\}}$, then $v_{\sigma} := (v_{\sigma_1},\dots,v_{\sigma_r})$ denotes the permutation of $v$ under $\sigma$. 

The so-called \emph{Young diagram of shape} $\lambda \in \mathcal{D}_{n,r}$ is a collection of left-aligned boxes, where the $j$-th row contains $\lambda_j$ boxes (see~\Cref{fig:Kostka}). The \emph{Kostka number} $K_{\lambda,\mu}$ with respect to $\lambda\in \mathcal{D}_{n,r}$ and $\mu \in \mathcal{D}_{\lambda}$, is the number of \emph{semi-standard Young tableaux of shape $\lambda$ and weight $\mu$} (see~\Cref{fig:Kostka}), i.e., it counts the number of ways to fill the Young diagram corresponding to $\lambda$ with integers in $(0:n-1)$ such that
\begin{enumerate}[label=\Roman*.]
    \item the integer $i$ can be used exactly $\mu_i$ times.
    \item the entries of the boxes do not decrease along rows (from left to right) and strictly increase along columns (top to bottom). 
\end{enumerate}
\begin{figure}
\begin{center}
\begin{tikzpicture}[scale=1]
\draw (0,0) rectangle (1,1);
\draw (1,0) rectangle (2,1);
\draw (0,-1) rectangle (1,0);
\node at (0.5,0.5) {1};
\node at (1.5,0.5) {2};
\node at (0.5,-0.5) {3};

\begin{scope}[xshift=3cm]
\draw (0,0) rectangle (1,1);
\draw (1,0) rectangle (2,1);
\draw (0,-1) rectangle (1,0);
\node at (0.5,0.5) {1};
\node at (1.5,0.5) {3};
\node at (0.5,-0.5) {2};
\end{scope}
\end{tikzpicture}
\end{center}
\caption{Semi-standard Young tableaux for the Kostka number $K_{\lambda,\mu}$ with $\lambda = (2,1,0)$ and $\mu = (1,1,1)$. Each tableau has three rows of boxes, with the first row containing $\lambda_1 = 2$ boxes, the second row containing $\lambda_2 = 1$ boxes, and the third row containing $\lambda_3 = 0$ boxes, which builds the Young diagram of shape $\lambda$. Then, there are exactly two possibilities to fill in $\mu_1$-times the integer $1$, $\mu_2$-times the integer $2$, and $\mu_3$-times the integer $3$, such that the entries of the tableaux do not decrease along rows and strictly increase along columns. \label{fig:Kostka}}
\end{figure}
We will use the following simple facts from \cite[7.10.5~Proposition]{stanley2023enumerative}.
\begin{lem}\label{lem:Kosta_number}
    Let $\lambda \in  \mathcal{D}_{n,r}$ and $\mu \in \mathcal{D}_\lambda$, then
    \begin{enumerate}[label=\roman*.]
        \item $K_{\lambda,\mu} > 0$ if and only if  $\lambda \trianglerighteq \mu$. 
        \item $K_{\lambda,\mu} = K_{\lambda - \lambda_r \mathbf{1}_r,\mu-\lambda_r \mathbf{1}_r}$. 
    \end{enumerate}
\end{lem}
Similarly, for $\lambda\in \mathcal{D}_{n,r}$ and $\gamma \in \mathcal{D}_{m,r}$ with $\lambda \leq \gamma$, the \emph{skew Young diagram of shape $\gamma/\lambda$} is defined as the difference between the Young diagram of shape $\gamma$ and the Young diagram of shape $\lambda$, i.e., the set of boxes that belong to $\gamma$, but not $\lambda$. The \emph{Littlewood–Richardson coefficients} $c_{\lambda,\mu}^\gamma$ with respect to $\lambda\in \mathcal{D}_{n,r}$, $\mu \in \mathcal{D}_{l,r}$ and $\gamma \in \mathcal{D}_{\lambda+\mu}$ can then be defined as the number of \emph{skew semi-standard Young tableaux of shape $\gamma/\lambda$ with weight $\mu$} (see~\Cref{fig:little}), where the weights are filled in such that
\begin{enumerate}[label=\Roman*.]
    \item the integer $i$ can be used exactly $\mu_i$ times.
    \item the entries of the boxes do not decrease along rows (from left to right) and strictly increase along columns (top to bottom). 
    \item The $r$-tuple $(w_1,\dots,w_{\sum_{i=1}^r \mu_i})$ that results from reading the tableau from right to left and top to bottom is a \emph{lattice word}, i.e., any $(w_1,\dots,w_j)$, $j \in (1:\sum_{i=1}^r \mu_i)$, contains the integer $i$ at least as often as the integer $i+1$ for all $i \in (1:\max(\mu)-1)$. 
\end{enumerate}

\begin{figure}
    \centering
 \begin{tikzpicture}[yscale=-1]

\draw (2,0) rectangle (3,1) node[pos=.5]{1};
\draw (3,0) rectangle (4,1) node[pos=.5]{1};

\draw (1,1) rectangle (2,2) node[pos=.5]{1};
\draw (2,1) rectangle (3,2) node[pos=.5]{2};

\draw (0,2) rectangle (1,3) node[pos=.5]{2};
\draw (1,2) rectangle (2,3) node[pos=.5]{3};

\draw (7,0) rectangle (8,1) node[pos=.5]{1};
\draw (8,0) rectangle (9,1) node[pos=.5]{1};

\draw (6,1) rectangle (7,2) node[pos=.5]{2};
\draw (7,1) rectangle (8,2) node[pos=.5]{2};

\draw (5,2) rectangle (6,3) node[pos=.5]{1};
\draw (6,2) rectangle (7,3) node[pos=.5]{3};

\end{tikzpicture}
    \caption{Skew semi-standard Young tableaux for the Littlewood-Richardson coefficient $c_{\lambda,\mu}^\gamma$ with $\lambda = (2,1,0)$, $\mu = (3,2,1)$ and $\gamma = (4,3,2)$. The Young diagram of shape $\lambda$ is as in \Cref{fig:Kostka}, which is missing from the Young diagram of shape $\gamma$ to establish the skew Young diagram of shape $\gamma/\lambda$. Then, there are exactly two possibilities to fill in $\mu_1$-times the integer $1$, $\mu_2$-times the integer $2$ and $\mu_3$-times the integer $3$, such that the entries of the tableaux do not decrease along rows, strictly increase along columns, and reading the tableaux from right to left and top to bottom creates a lattice word. The lattice word property in the left tableau means that $(1)$, $(1,1)$, $(1,1, 2)$, $(1,1,2,1)$, $(1,1,2,1,3)$ and $(1,1,2,1,3,2)$ contain at least as many $1$s as $2$s and at least as many $2$s as $3$s, and similarly for the right tableau.}
    \label{fig:little}
\end{figure}
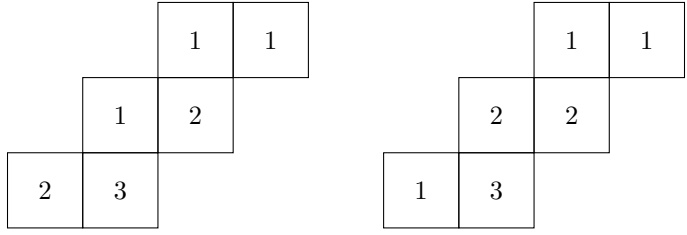
Kostka numbers and Littlewood-Richardson coefficients originated from the study of so-called Schur polynomials, where the \emph{Schur polynomial to $\lambda \in \mathcal{D}_{n,r}$ in variables in $x \in \mathds{C}^r$}, is defined as the ratio
\begin{equation}
    s_\lambda(x_1,\dots,x_r) := \frac{a_{(\lambda_1+r-1,\lambda_2+r-2,\dots,\lambda_r+0)}(x_1,\dots,x_r)}{a_{(r-1,r-2,\dots,0)}(x_1,\dots,x_r)} \label{eq:schur_poly}
\end{equation}
where the \emph{alternating polynomial $a_\mu$ in variables $x \in \mathds{C}^r$} for $\mu \in (1:m)^r$ is given by
\begin{equation*}
    a_{\mu}(x_1,\dots,x_r) := \begin{vmatrix}
        x_1^{\mu_r} & \dots & x_r^{\mu_r}\\
        x_1^{\mu_{r-1}} & \dots & x_r^{\mu_{r-1}}\\
        \vdots & & \vdots \\
        x_1^{\mu_1} & \dots & x_r^{\mu_1}
    \end{vmatrix}.
\end{equation*}
Using the notation $x^p := \prod_{i=1}^r x_i^{p_i}$ for $x \in \mathds{C}^r$ and $p \in (1:m)^r$, it follows from the Leibniz formula for determinants that 
\begin{equation}
a_{\mu}(x) = \begin{cases}
    \sum_{p \in \permut{\mu}} \sign_\mu(p) x^p & \text{if $\mu$ has distinct values}, \\
    0 & \text{else}.
\end{cases}  \label{eq:alt_poly_leibniz}  
\end{equation}
By \cite[A1.5.3 Theorem and Eq. (7.35)]{stanley2023enumerative}, it holds then for $\lambda\in \mathcal{D}_{n,r}$, $\mu \in \mathcal{D}_{l,r}$ that 
\begin{equation}
    s_\lambda(x) s_{\mu}(x) = \sum_{\gamma \in \mathcal{D}_{\lambda+\mu}} c^{\gamma}_{\lambda,\mu} s_\gamma(x) \label{eq:LRrule}
\end{equation}
and
\begin{equation}
    s_\lambda(x) = \sum_{\mu \in \mathcal{D}_\lambda} K_{\lambda,\mu} m_\mu(x), \label{eq:schur_kostka}
\end{equation}
with \emph{monomial symmetric function} $m_\mu(x) := \sum_{p \in \permut{\mu}} x^p$.

\section{Main Results}\label{sec:results}
In \cref{lem:k_pos_Hankel}, we observed that $k$-positivity of Hankel matrices $H^g(1,M,N)$ can be verified efficiently by considering (consecutive) minors only within $H^g(1,k,M+N-k)$. The main objective of this work is to extend this result to the setting of \emph{$k$-sign consistent} Hankel matrices and operators.
\subsection{Row-Consecutive Minor Decomposition}
Central to our analysis is the following decomposition of arbitrary minors into nonnegative integer linear combinations of row-consecutive minors.
\begin{thm} \label{thm:skip_minor_decomp}
    Let $A =(a_{ij}) \in \mathds{C}^{M \times n}$, $M \geq n$, $s \in \mathcal{I}_{M,n}$ and $\lambda \in \mathcal{D}_{s_{n},n}$ be defined by $\lambda_i := s_{n-i+1}-(n-i+1)$. Then,  
    \begin{equation}
    |A_{(s,:)}| = \sum_{\mu \in \mathcal{D}_{\lambda}} K_{\lambda,\mu} D_{\mu} \label{eq:skip_minor_decomp}
    \end{equation}
    where
    \begin{equation}
        D_\mu := \sum_{p \in \permut{\mu}} \begin{vmatrix}
            a_{p_{1}+1,1} & \dots & a_{p_{n}+1,n}\\
            \vdots &  & \vdots\\
            a_{p_{1}+n,1} & \dots & a_{p_n+n,n}
        \end{vmatrix}
    \end{equation}
    and $K_{\lambda,\mu}$ are the Kostka numbers associated with \(\lambda\) and \(\mu\).
\end{thm}
\begin{rem}\label{rem:skip_minor_decomp_n_2}
It suffices to consider the case \(s_1 = 1\) in \cref{thm:skip_minor_decomp}; all other cases follow by replacing \(A\) with \(A_{((s_1:M),:)}\). This is consistent with the identity \( K_{\lambda,\mu} = K_{\lambda - \lambda_n \mathbf{1}_n,\; \mu - \lambda_n \mathbf{1}_n} \) from \cref{lem:Kosta_number}. For the special case \(n=2\), the Young diagram of shape \((s_2-2,0)\) consists of a single row, implying \(K_{\lambda,\mu}=1\). Hence,
\begin{equation}
    \begin{vmatrix}
        a_{1,1} & a_{1,2}\\
        a_{s_2,1} & a_{s_2,2}
    \end{vmatrix} = \sum_{i=0}^{\floor{\frac{s_2-2}{2}}} D_{(s_2-2-i,i)} = \sum_{i=0}^{s_2-2}  \begin{vmatrix}
        a_{1+i,1} & a_{s_2-i-1,2}\\
        a_{2+i,1} & a_{s_2-i,2}
    \end{vmatrix}.
\end{equation}
In other words, for $n=2$, \cref{thm:skip_minor_decomp} amounts to a telescope-sum expansion. 
\end{rem}
We prove \cref{thm:skip_minor_decomp} by first establishing its validity for the specialization \(A = V_M(x)\), which is a direct reformulation of \cref{eq:schur_kostka}.
\begin{lem}\label{lem:vander}
    Let $S := V_M(x) \in \mathds{C}^{M \times n}$, $M \geq n$, for some $x \in \mathds{C}^n$,  $s \in \mathcal{I}_{M,n}$ and $\lambda \in \mathcal{D}_{s_{n},n}$ be given by $\lambda_i := s_{n-i+1}-(n-i+1)$. Then, 
  \begin{equation} \label{eq:skipped_minor_poly}
    |S_{(s,:)}| = \sum_{\mu \in \mathcal{D}_\lambda} K_{\lambda,\mu} d_{\mu}(x)
    \end{equation} 
   where
     \begin{equation}
        d_\mu(x) := \sum_{p \in \permut{\mu}}  \begin{vmatrix}
            x_1^{p_{1}} & \dots & x_n^{p_{n}}\\
            \vdots &  & \vdots\\
            x_1^{p_{1}+n-1} & \dots & x_n^{p_n+n-1,n}
        \end{vmatrix} = |S_{(1:n,:)}| m_{\mu}(x). \label{eq:d_mu_poly}
    \end{equation}
        and $K_{\lambda,\mu}$ are the Kostka numbers with respect to $\lambda$ and $\mu$. 
\end{lem}
\begin{proof}
    Since equality in \cref{eq:d_mu_poly} follows from the multi-linearity of the determinant, and since $\frac{|S_{(s,:)}|}{|S_{((1:n),:)}|}$ is a Schur polynomial, our claim can be restated as
    \begin{equation*}
        \frac{|S_{(s,:)}|}{|S_{((1:n),:)}|} = \frac{1}{{|S_{((1:n),:)}|}} \sum_{\mu \in \mathcal{D}_\lambda} K_{\lambda,\mu} m_\mu(x),
    \end{equation*}
which is identical to \cref{eq:schur_kostka}.
\end{proof}
\begin{proof}[\cref{thm:skip_minor_decomp}]
Let $x \in \mathds{C}^{M}$ be such that all $x_i$ are distinct from each other. Then, $\rk(V_M(x)) = M$ and there exists a $C \in \mathds{C}^{M\times n}$ such that $A = V_M(x) C$. By the multi-linearity of the determinant and \cref{lem:vander}, it follows then that
\begin{align*}
|A_{(s,:)}| &=  \sum_{j_1=1}^{M} \cdots \sum_{j_n=1}^M \left[ |{V_M(x_{\{j_1,\dots,j_n\}})}_{(s,:)}| \prod_{i=1}^n c_{j_i i} \right] = \sum_{j_1=1}^{M} \cdots \sum_{j_n=1}^M \left[ \sum_{\mu \in \mathcal{D}_\lambda} K_{\lambda,\mu} d_{\mu}(x_{\{j_1,\dots,j_n\}}) \prod_{i=1}^n c_{j_i i} \right] \\
&= \sum_{\mu \in \mathcal{D}_\lambda} K_{\lambda,\mu}  \sum_{j_1=1}^{M} \cdots \sum_{j_n=1}^M \left[  d_{\mu}(x_{\{j_1,\dots,j_n\}}) \prod_{i=1}^n c_{j_i i} \right],
\end{align*}
where $d_{\mu}(x_{\{j_1,\dots,j_n\}})$ are defined as in \cref{lem:vander}. Since all the determinants that compose $d_{\mu}(x_{\{j_1,\dots,j_n\}})$ result from the same columns as ${V_M(x_{\{j_1,\dots,j_n\}})}_{(s,:)}$, we are allowed to revert our initial expansion such that
\begin{align*}
    \sum_{j_1=1}^{M} \cdots \sum_{j_n=1}^M \left[  d_{\mu}(x_{\{j_1,\dots,j_n\}}) \prod_{i=1}^n c_{j_i i} \right] = \sum_{p \in \permut{\mu}} \begin{vmatrix}
            a_{p_{1}+1,1} & \dots & a_{p_{n}+1,n}\\
            \vdots &  & \vdots\\
            a_{p_{1}+n,1} & \dots & a_{p_n+n,n}
        \end{vmatrix} = D_\mu
\end{align*}
and our claim follows. 
\end{proof}
Although \cref{thm:skip_minor_decomp} is a straightforward representation-theoretic identity, the authors are not aware of any prior explicit appearance of this result in the literature.

\subsection{$k$-Sign Consistency of Hankel Matrices}
An immediate consequence of \cref{thm:skip_minor_decomp} together with the nonnegativity of Kostka numbers is the following certificate for $k$-sign consistency. 
\begin{cor}\label{cor:k_sign_row_consec} Let \(A = (a_{ij}) \in \mathds{C}^{M \times n}\) with \(M \ge n\). If all row-consecutive \(k\)-minors of \(A\) share the same (strict) sign, then \(A\) is (strictly) \(k\)-sign consistent. \end{cor} 
Although \cref{cor:k_sign_row_consec} provides a sufficient certificate, it is typically inefficient: it requires computing far more minors than directly evaluating all $k$-minors, and it is often overly restrictive. For our purposes, however, i.e., the study of \(k\)-sign consistent Hankel matrices, this criterion is surprisingly close to the favorable checks that we want to show. The case \(k=2\) already illustrates this phenomenon.
\begin{example}\label{exm:hankel_k_2}
Let us apply \cref{thm:skip_minor_decomp} to any two columns of
\begin{equation}
   A := H^g(1,M,N) = \begin{bmatrix}
        g_1 & g_{i+1} & \dots & g_{N}\\
        g_{2} & g_{i+2} & \dots & g_{N+1}\\
        \vdots & \vdots & & \vdots \\
        g_{M} & g_{M+1} & \dots & g_{M+N-1}
    \end{bmatrix}, \; M,N \geq 2.
\end{equation}
Similar to \cref{rem:skip_minor_decomp_n_2}, it sufficies to consider the $2$-minors $|H^g(1,s,r)|$, $s \in \mathcal{I}_{M,2}$ and 
$r \in \mathcal{I}_{N,2}$ with $s_1 =r_1 = 1$, because all other cases follow by deleting the first $s_1-1$ rows and $r_1-1$ columns of $A$. 

Thus, by \cref{rem:skip_minor_decomp_n_2}
\begin{align}
     |A_{(s,r)}| &= \begin{vmatrix}
            g_{1} & g_{r_2}\\
            g_{s_2} & g_{s_2+r_2-1}
        \end{vmatrix} = \sum_{i=0}^{s_2-2}\begin{vmatrix}
            g_{i+1} & g_{s_2-2-i+r_2}\\
            g_{i+2} & g_{s_2-1-i+r_2}
        \end{vmatrix},  \label{eq:Hankel_decomp_k_2}
\end{align}
which shows that all $2$-minors of $A$ can be expressed as a linear combination of $2$-minors of the form 
\begin{equation*}
    b^g_{i_1,i_2} := \begin{vmatrix}
        g_{i_1} & g_{i_2}\\
        g_{i_1+1} & g_{i_2+1}
    \end{vmatrix} , \; i \in \mathcal{I}_{M+N-2,2},
\end{equation*}
with a negative sign if $i+1 > s_2-2-i+r_2$, i.e., $2i>s_2+r_2-3$. If $s_2 \leq r_2+1$, then every term in our expansion is a regular $2$-minor of $H^g(1,2,M+N-2)$. Moreover, if $s_2 \geq r_2+1$, then
\begin{align*}
   |A_{(s,r)}| &= \sum_{i=0}^{s_2-2}\begin{vmatrix}
            g_{i+1} & g_{s_2-2-i+r_2}\\
            g_{i+2} & g_{s_2-1-i+r_2}
        \end{vmatrix}
        = \sum_{i=0}^{\floor{\frac{s_2+r_2-3}{2}}} \begin{vmatrix}
            g_{i+1} & g_{s_2-2-i+r_2}\\
            g_{i+2} & g_{s_2-1-i+r_2}
        \end{vmatrix} - \sum_{i= \floor{\frac{s_2+r_2-3}{2}}+1}^{s_2-2} \begin{vmatrix}
          g_{s_2-2-i+r_2} &  g_{i+1} \\
          g_{s_2-1-i+r_2} &  g_{i+2} 
        \end{vmatrix} \\ &= \sum_{i=0}^{\floor{\frac{s_2+r_2-3}{2}}} \begin{vmatrix}
            g_{i+1} & g_{s_2-2-i+r_2}\\
            g_{i+2} & g_{s_2-1-i+r_2}
        \end{vmatrix} - \sum_{i= r_2-1}^{\ceil{\frac{s_2+r_2-3}{2}}-1} \begin{vmatrix}
             g_{i+1} & g_{s_2-2-i+r_2}\\
            g_{i+2} & g_{s_2-1-i+r_2}
        \end{vmatrix} = \sum_{i=0}^{r_2-2} \begin{vmatrix}
            g_{i+1} & g_{s_2-2-i+r_2}\\
            g_{i+2} & g_{s_2-1-i+r_2}
        \end{vmatrix},
\end{align*}
where the last equality removes a zero $2$-minor of the first sum in case that $s_2+r_2-3$ is even. Hence, all $2$-minors of $A$ are nonnegative linear integer combinations of the $2$-mionors $b^g_{i_1,i_2}$, $i \in \mathcal{I}_{M+N-2,2}$, or equivalently, of the $2$-minors of $H^g(1,2,M+N-2)$. 

This shows that $H^g(1,M,N)$ is (strictly) $2$-sign consistent if the reshaped Hankel matrix $H^g(1,2,M+N-2)$ is (strictly) $2$-sign consistent. In case of $\Hank{g}$, this is also a necessary condition, since all $2$-minors of $H^g(1,2,2N-2)$ are two $2$-minors of $H^g(1,2N-2,2N-2)$.
\end{example}

In the following, we extend the conclusions of \cref{exm:hankel_k_2} to the cases with \(k>2\). Let \(A = H^g(1,M,N)\), \(s \in \mathcal{I}_{M,k}\), and \(r \in \mathcal{I}_{N,k}\). Then \cref{eq:skip_minor_decomp} specializes to
\begin{equation}
    |A_{(s,r)}| = \sum_{{\mu \in \mathcal{D}_\lambda}} K_{\lambda,\mu} D_{\mu,r}
\label{eq:Hankel_Kosta_decomp}
\end{equation}
where $\lambda_i = s_{k-i+1} -(k-i+1)$ and 
\begin{equation}
        D_{\mu,r} := \sum_{p \in \permut{\mu}} b^g_{(p_1+r_1,\dots,p_k+r_k)} 
        \label{eq:D_mu_c}
\end{equation}
with
\begin{equation*}
     b^g_{(i_1,\dots,i_k)} := \begin{vmatrix}
        g_{i_1} & \dots & g_{i_k} \\
        \vdots &  & \vdots\\
            g_{i_{1}+k-1} & \dots & g_{i_k+k-1}
    \end{vmatrix}, \; i \in (1:M+N-k)^k. 
\end{equation*}
As in \cref{exm:hankel_k_2}, this shows that every \(k\)-minor of \(A\) is a linear combination of the \(k\)-minors of the reshaped Hankel matrix \(H^g(1,k,M+N-k)\). Thus, it remains to show that -- after possible cancellations -- each such \(k\)-minor appears with a nonnegative coefficient in \cref{eq:Hankel_Kosta_decomp}.

To this end, we identify which \(\mu \in \mathcal{D}_\lambda\) (up to column permutations) contribute to the same minor \(b^g_{(i_1,\dots,i_k)}\). For fixed $\mu^\ast \in \mathcal{D}_\lambda$ and $p^\ast \in \permut{\mu^\ast}$, the $k$-minor $b^g_{(p_1^\ast+r_1,\dots,p_k^\ast+r_k)}$ appears in \cref{eq:Hankel_Kosta_decomp} for other choices of $\mu \in \mathcal{D}_\lambda$ and $p \in \permut{\mu}$, up to column permutation, if and only if $$r + p \in \permut{v^\ast},$$ where $v^\ast := (r + p^\ast)$. Define $\mathcal{Q}(v^\ast) := \{q \in \permut{v^\ast}: q - r \geq 0 \}$, then the set of all such $\mu$ is given by 
\begin{equation}
\{(q - r)^\downarrow: q \in \mathcal{Q}(v^\ast) \} \label{eq:set_all_mu}
\end{equation}
with corresponding $k$-minors $b^g_{q}$ appearing in $D_{(q-r)^\downarrow,r}$ in \cref{eq:D_mu_c}. Since \eqref{eq:set_all_mu} is nonempty if and only if \({v^\ast}^\uparrow \in \mathcal{Q}(v^\ast)\), and since we only need to consider \(v^\ast\) with distinct entries (otherwise \(b^g_{v^\ast}=0\)), we may use the identity
\begin{equation*}
    \forall q \in \mathcal{Q}(v^\ast): b^g_{q} = \sign_{{v^\ast}^\uparrow}(q)b^g_{{v^\ast}^\uparrow}.
\end{equation*}
Thus,
\begin{equation}
    |A_{(s,r)}| = \sum_{{v^\ast \in V^\ast}} \sum_{ q \in \mathcal{Q}(v^\ast)} \sign_{{v^\ast}}(q) K_{\lambda,(q-r)^\downarrow} b^g_{v^\ast},\label{eq:minor_decomp_nonneg}
\end{equation}
where $$V^\ast := \{(p+r)^\uparrow: p \in \permut{\mu}, \; \mu \in \mathcal{D}_\lambda, p+r \text{ has distinct entries}\}.$$ Since each \(b^g_{v^\ast}\) is a \(k\)-minor of \(H^g(1,k,M+N-k)\), the desired nonnegative integer combination exists provided that 
\begin{equation*}
  \forall v^\ast \in V^\ast:  \sum_{ q \in \mathcal{Q}(v^\ast)} \sign_{{v^\ast}}(q) K_{\lambda,(q-r)^\downarrow} \geq 0.
\end{equation*}
This inequality follows from the next lemma.
\begin{lem}\label{lem:Kostka_zero_sum}
Let \(\lambda \in \mathcal{D}_{M,k}\), \(r \in \mathcal{I}_{N,k}\), and

\[
v^\ast \in
\{\, (p+r)^\uparrow :
    p \in \permut{\mu},\;
    \mu \in \mathcal{D}_\lambda,\;
    (p+r)^\uparrow_i > (p+r)^\uparrow_{i-1}
    \text{ for all } i \in (2:k)
\}.
\]
Then
\begin{equation}\label{eq:kostka_sum_null}
\sum_{q \in \mathcal{Q}(v^\ast)}
\sign_{v^\ast}(q)\,
K_{\lambda,(q-r)^\downarrow}
=
\sum_{p \in \mathcal{P}(v^\ast)}
\sign_r(p)\,
K_{\lambda,(v^\ast - p)^\downarrow}
=
c^{\gamma^\ast}_{\lambda,\varepsilon},
\end{equation}
where
\[
\mathcal{Q}(v^\ast)
:= \{\, q \in \permut{v^\ast} : q - r \ge 0 \,\},\qquad
\mathcal{P}(v^\ast)
:= \{\, p \in \permut{r} : v^\ast - p \ge 0 \,\},
\]
and
\[
\gamma_i^\ast := v^\ast_{k-i+1} - k + i,
\qquad
\varepsilon_i := r_{k-i+1} - k + i,
\qquad
i \in (1:k).
\]

\end{lem}
\begin{proof}
We begin by showing the first equality in \cref{eq:kostka_sum_null}. Observe that
\begin{equation*}
    \{(q - r)^\downarrow: q \in \mathcal{Q}(v^\ast) \} = \{(v^\ast - p)^\downarrow: p \in \mathcal{P}(v^\ast) \},
\end{equation*}
since sorting removes any dependence on the order of the entries. By the assumption on \(v^\ast\), each of the sets
\begin{equation*}
\{q-r: q \in \mathcal{Q}(v^\ast)\} \quad \text{and} \quad \{v^\ast-p: p \in \mathcal{P}(v^\ast)\}
\end{equation*}
consists of distinct elements, i.e., for every $p \in \mathcal{P}(v^\ast)$ there exist unique $q \in \mathcal{Q}(v^\ast)$ and $\sigma \in \permut{(1:k)}$ such that $$(v^\ast - p)_{\sigma} = v^\ast_{\sigma} -p_{\sigma} = q - r.$$ As this equation is fulfilled by $p_{\sigma} = r$ and $q = v^\ast_{\sigma}$, we conclude that $\sign_{v^\ast}(q) = \sign_r(p)$. This proves the first equality in \cref{eq:kostka_sum_null}.  
Next, define the polynomial
\begin{equation}
  F(x) :=  s_\lambda(x) \sum_{p \in \permut{r}} \sign_r(p) x^{p} = \sum_{p \in \permut{r}} \sum_{\mu \in \mathcal{D}_\lambda}    \sum_{q \in \permut{\mu}} \sign_r(p) K_{\lambda,\mu} x^{q+p}, \label{eq:poly_x_v}
\end{equation}
where equality follows by \cref{eq:schur_kostka}. Since $\{(v^\ast-p)^\downarrow: p \in \mathcal{P}(v^\ast)\} \subset \mathcal{D}_\lambda$ and $v^\ast-p \in \permut{(v^\ast-p)^\downarrow}$, we see that $x^{q+p} = x^{v^\ast}$ for some $q \in \permut{\mu}$ if and only if $p \in \mathcal{P}(v^\ast)$ and $\mu = (v^\ast-p)^\downarrow$. Therefore, $\sum_{p \in \mathcal{P}(v^\ast)} \sign_r(p) K_{\lambda,(v^\ast - p)^\downarrow}$ is exactly the coefficient of $x^{v^\ast}$ in $F(x)$. By \cref{eq:alt_poly_leibniz,eq:schur_poly}, it follows that
\begin{equation*}
    \sum_{p \in \permut{r}} \sign_r(p) x^{p} = a_{r}(x) = s_\varepsilon(x) a_{(0,1,\dots,k-1)} 
\end{equation*}
and, therefore, using \cref{eq:LRrule,eq:alt_poly_leibniz},
\begin{align*}
 F(x) &= s_\lambda(x) s_\varepsilon(x) a_{(0,1,\dots,k-1)}(x) =
    \sum_{\gamma \in \mathcal{D}_{\lambda+\varepsilon}} c^{\gamma}_{\lambda,\varepsilon} s_\gamma(x) a_{(0,1,\dots,k-1)}(x)\\ 
    &= 
    \sum_{\gamma \in \mathcal{D}_{\lambda+\varepsilon}} c^{\gamma}_{\lambda,\varepsilon} a_{(\gamma_k,\gamma_{k-1}+1,\dots,\gamma_1+k-1)}.
\end{align*}
Since \(v^\ast\) has strictly increasing entries, the monomial \(x^{v^\ast}\) appears in \(F(x)\) if and only if \[ a_{(\gamma_k,\gamma_{k-1}+1,\dots,\gamma_1+k-1)} = a_{v^\ast} \] for some \(\gamma \in  \mathcal{D}_{\lambda+\varepsilon}\) with \(c^{\gamma}_{\lambda,\varepsilon} \neq 0\). Thus, the coefficient of \(x^{v^\ast}\) in \(F(x)\) is exactly \(c^{\gamma^\ast}_{\lambda,\varepsilon}\), proving \[ \sum_{p \in \mathcal{P}(v^\ast)} \sign_r(p)\, K_{\lambda,(v^\ast - p)^\downarrow} = c^{\gamma^\ast}_{\lambda,\varepsilon}. \]
\end{proof} 
We illustrate \cref{lem:Kostka_zero_sum} with two numerical examples. 
\begin{example}
    Let $\lambda = (4,1,0)$, $r = (1,2,4)$, and consequently $\varepsilon = r^\downarrow - (2,1,0) = (2,1,0)$. We consider the following two cases:
\begin{enumerate}
    \item \textbf{Case 1 \(v^\ast = (1,4,7)\):} Here $\gamma^\ast = (5,3,1)$, so by \cref{lem:Kostka_zero_sum} it is possible that \cref{eq:kostka_sum_null} is positive. $Q(v^\ast)$ consists of 
    \begin{equation*}
    (0,2,3) = v^\ast-r \qquad \text{and} \qquad( 0,5,0) = (1,7,4)-r,
    \end{equation*}
    where $K_{\lambda,(5,0,0)} = 0$ by \cref{lem:Kosta_number} and $K_{\lambda,(3,2,0)} = 1$ with corresponding $\sign_{v^\ast}(v^\ast) = 1$. Hence, $\sum_{q \in \mathcal{Q}(v^\ast)} \sign_{v^\ast}(q)\, K_{\lambda,(q-r)^\downarrow} = 1$, which which agrees with the Littlewood–Richardson coefficient $c^{\gamma^\ast}_{\lambda,\varepsilon} = 1$.  

     \item \textbf{Case 2 \(v^\ast = (2,4,6)\):} In this case, $\gamma^\ast = (4,3,2)$, so again \cref{lem:Kostka_zero_sum} allows the possibility of a positive value. However, one can show that $c^{\gamma^\ast}_{\lambda,\varepsilon} = 0$. This can also be verified by noticing that $Q(v^\ast)$ is comprised of 
     \begin{equation*}
         (5,0,0) = (6,2,4) - r, \qquad (3,0,2) = (4,2,6) - r \qquad \text{and} \qquad (1,2,2) = v^\ast -r,
     \end{equation*}
     with corresponding
     \begin{equation*}
         \sign_{v^\ast}(6,2,4) K_{\lambda,(3,2,0)} = 1, \quad \sign_{v^\ast}(4,2,6) K_{\lambda,(4,1,0)} = -1, \quad \text{and} \quad \sign_{v^\ast}(v^\ast) K_{\lambda,(2,2,1)} = 2.
     \end{equation*}
     Thus, $\sum_{q \in \mathcal{Q}(v^\ast)} \sign_{{v^\ast}}(q) K_{\lambda,(q - r)^\downarrow} = 0$. 
\end{enumerate}
\end{example}
\begin{rem} \Cref{eq:kostka_sum_null} also appears in a different form in \cite[Corollary~1]{shrivastava2023littlewood}, namely 
\begin{equation}\label{eq:Kosta_sum_zero_alt} c^{\gamma}_{\lambda,\varepsilon} = \sum_{\substack{\eta \in \permut{(1,\dots,k)} \\ \psi \trianglelefteq \varepsilon}} \sign_{(1,\dots,k)}(\eta)\, K_{\varepsilon,\psi}, 
\end{equation}
where only nonnegative $\psi = ( (\gamma + (k-1,\dots,1,0))_{\eta} + \lambda - (k-1,\dots,1,0))^\downarrow$ are permitted. To see the equivalence with \cref{eq:kostka_sum_null}, note first that \( c^{\gamma}_{\lambda,\varepsilon} = c^{\gamma}_{\varepsilon,\lambda} \) by \cref{eq:LRrule}, so one may interchange the roles of \(\lambda\) and \(\varepsilon\) in order to express the sum using \(\lambda\)-dependent Kostka numbers, as in \cref{eq:kostka_sum_null}. For the choices of \(v^\ast\) and \(r\) (and corresponding \(\gamma\) and \(\varepsilon\)) in \cref{lem:Kostka_zero_sum}, we have
\begin{equation*}
    \gamma + (k-1,\dots,1,0) = {v^\ast}^\downarrow \quad \text{and} \quad \varepsilon + (k-1,\dots,1,0) = r^\downarrow
\end{equation*}
so that $\psi = {(({v^\ast}^\downarrow)_\eta - r^\downarrow)}^\downarrow$. Moreover, since \({v^\ast}^\downarrow - r^\downarrow \in \mathcal{D}_{\lambda}\), the condition \(\psi \trianglelefteq \lambda\) may be omitted by \cref{lem:Kosta_number}. Therefore, \cref{eq:Kosta_sum_zero_alt} becomes
\begin{align*}
    c^{\gamma}_{\lambda, \varepsilon} = c^{\gamma}_{\varepsilon, \lambda} &= \sum_{\substack{\eta \in \permut{(1,\dots,k)},\\ {({v^\ast}^\downarrow)_\eta-r^\downarrow }\geq 0} } \sign_{(1,\dots,k)}(\eta)\, K_{\lambda, (({v^\ast}^\downarrow)_{\eta} - r^\downarrow)^\downarrow}
    = \sum_{ q \in \mathcal{Q}({v^\ast}^\downarrow)} \sign_{{v^\ast}^\downarrow}(q) K_{\lambda,(q-r)^\downarrow} \\
    &= \sum_{ q \in \mathcal{Q}(v^\ast)} \sign_{{v^\ast}}(q) K_{\lambda,(q-r)^\downarrow},
\end{align*}
where the final equality uses the fact that \(v^\ast\) has distinct entries. Although the resulting identity is the same, our derivation differs substantially from that of~\cite{shrivastava2023littlewood}. Our approach is motivated by applications to Hankel matrices and requires significantly fewer prerequisites.
\end{rem}
In conclusion, since Littlewood–Richardson coefficients are nonnegative by definition, \cref{eq:minor_decomp_nonneg} shows that each minor \(|A_{(s,r)}|\) is a nonnegative linear combination of the minors $b^g_{(i_1,\dots,i_k)}, \; i \in \mathcal{I}_{M+N-k,k}$. Hence, if all $b^g_{(i_1,\dots,i_k)}$ share the same sign, then so do all $|A_{(s,r)}|$. The strict case follows similarly, since by \begin{equation}\label{eq:strict_LR_coeff} c^{\gamma(\lambda^\uparrow + r)}_{\lambda,\varepsilon} = c^{\lambda+\varepsilon}_{\lambda,\varepsilon} \ge 1 \end{equation} at least one strictly positive coefficient appears in \cref{eq:minor_decomp_nonneg}. 
This is summarized in the following main result.
\begin{thm}\label{thm:row_consec_Hankel} Let \(g \in \mathds{C}^{\mathds{Z}}\), $A := H^g(1,M,N) \in \mathds{C}^{M \times N}$ and $ k \le \min\{M,N\}$. For \(s \in \mathcal{I}_{M,k}\) and \(r \in \mathcal{I}_{N,k}\), define \[ b^g_{(i_1,\dots,i_k)} := \begin{vmatrix} g_{i_1} & \dots & g_{i_k} \\ \vdots & & \vdots \\ g_{i_1+k-1} & \dots & g_{i_k+k-1} \end{vmatrix}, \qquad (i_1,\dots,i_k) \in (1:M+N-k)^k \] Then \begin{equation}\label{eq:Hankel_minor_formula} |A_{(s,r)}| = \sum_{v^\ast \in V^\ast} c^{\gamma(v^\ast)}_{\lambda,\varepsilon}\, b^g_{v^\ast}, \end{equation} where \[ \lambda_i := s_{k-i+1} - (k-i+1), \qquad \varepsilon_i := r_{k-i+1} - k + i, \qquad \gamma_i(v^\ast) := v^\ast_{k-i+1} - k + i, \qquad i \in (1:k) \]  and \[ V^\ast := \bigl\{ (p+r)^\uparrow : p \in \permut{\mu},\; \mu \in \mathcal{D}_\lambda,\; (p+r)^\uparrow_i > (p+r)^\uparrow_{i-1} \text{ for all } i \in (2:k) \bigr\}. \] In particular, if the reshaped Hankel matrix \(H^g(1,k,M+N-k)\) is (strictly) \(k\)-sign consistent, i.e., all minors \(b^g_{(i_1,\dots,i_k)}\), \(i \in \mathcal{I}_{M+N-k,k}\), share the same (strict) sign, then \(A\) is (strictly) \(k\)-sign consistent. 
\end{thm}
Note that \cref{thm:row_consec_Hankel} provides only a sufficient condition, since it requires (strict) \(k\)-sign consistency of a reshaped Hankel matrix. In the finite-dimensional setting this condition may be stronger than necessary (see the circulant matrix case below). However, in the infinite-dimensional operator case, the situation changes fundamentally: the condition becomes not only sufficient but also necessary. This leads us to the following analogue of \cref{cor:k_pos_Hankel}. 
\begin{cor}\label{cor:Hankel_op_k_sign}
Let $g \in \mathds{C}^\mathds{Z}$. Then, the following are equivalent:
\begin{enumerate}
    \item $\mathcal{H}^g$ is (strictly) $k$-sign consistent. 
    \item For all $N \geq k$: $H^g(1,k,N)$ is (strictly) $k$-sign consistent. 
\end{enumerate}
\end{cor}
\begin{proof}
   By definition, \(\mathcal{H}^g\) is (strictly) \(k\)-sign consistent if and only if \(H^g(1,N,N)\) is (strictly) \(k\)-sign consistent for every \(N \ge k\). However, by \cref{thm:row_consec_Hankel}, this requirement is already satisfied provided that \(H^g(1,k,2N-k)\) is (strictly) \(k\)-sign consistent for all \(N \ge k\). Thus the two conditions are equivalent.
\end{proof}
At this point, it is worth noting that the class of \(k\)-sign consistent Hankel matrices (and operators) is strictly larger than the class of \(k\)-positive Hankel matrices (and operators).  
Indeed, \(k\)-sign consistency does not require \(1\)-sign consistency.  
The following example illustrates this.

\begin{example}\label{ex:Hankel_2_sign_not_2_pos}
Let \(g_t = \alpha_1^{t-1} + k\, \alpha_2^{t-1}\) for \(t \ge 1\), where  
\(k, \alpha_1, \alpha_2 \in \mathds{R} \setminus \{0\}\).  
Then, for all \(t,j \in \mathds{Z}_{>0}\),
\begin{equation}\label{eq:2_minor_impulse}
    b^g_{(t,t+j)}
    =
    p\, \alpha_1^{t-1} \alpha_2^{t-1}
    \bigl(\alpha_1^j - \alpha_2^j\bigr)
    (\alpha_1 - \alpha_2).
\end{equation}
Thus \(H^g(1,M,N)\) is strictly \(2\)-sign consistent for any \(M,N \ge 2\)  
if and only if \(\alpha_1, \alpha_2 > 0\).  However, strict \(1\)-sign consistency fails in general.  
For example, choosing \(\alpha_1 = 1\), \(\alpha_2 = 0.5\), and \(p = -3\) yields
\[
    (g_1, g_2, g_3) = (-2,\,-0.5,\,0.25),
\]
which changes sign.  
Hence \(H^g(1,M,N)\) is \(2\)-sign consistent but not \(1\)-sign consistent.
\end{example}
Unfortunately, unlike \cref{cor:k_pos_Hankel}, there is in general no further simplification that would allow us to check only, for example, consecutive \(k\)-minors of \(H^g(1,k,M+N-k)\).  Indeed, in \cref{ex:Hankel_2_sign_not_2_pos}, if one were to impose strict sign consistency solely on consecutive \(2\)-minors (i.e., by choosing \(j=1\) in \cref{eq:2_minor_impulse}), then the choice \(\alpha_1, \alpha_2 \in \mathds{R}_{<0}\) would also satisfy this restricted condition, despite being incorrect for full \(2\)-sign consistency.

Yet, it should be noted that \cref{thm:row_consec_Hankel} enables the use of \cref{lem:pena_bijection}. In particular, if \(H^g(1,k,k)\) is invertible, then the (strict) \(k\)-sign consistency of  
\(H^g(1,M,N)\), \(M,N \ge k\), can be verified by checking the (strict) total positivity of
\[
    K_k^\transp\, H^g(1,k,k)^{-1}\, H^g(k+1,k,M+N-k),
\]
where \(K_k\) is defined in \cref{lem:pena_bijection}.  
This verification can be performed efficiently using, for example, \cref{prop:row_col_initial_minors_strict_tot_pos}.  
Then, employing Dodgson condensation as in \cite{carter2021complexity}, the overall computational complexity reduces to \(\mathcal{O}\!\left(k^{2}(M+N)\right)\) whenever \(k \le N + M - k\).

\subsubsection{Toeplitz Matrices}
Since reversing the column order of a Toeplitz matrix produces a Hankel matrix without affecting its (strict) \(k\)-sign consistency, direct analogues of \cref{thm:row_consec_Hankel,cor:Hankel_op_k_sign} hold in the Toeplitz setting.
\begin{cor}\label{cor:toep} Let \(g \in \mathds{C}^{\mathds{Z}}\), and \(A := T^g(t,M,N) \in \mathds{C}^{M \times N} \) with \(k \le \min\{M,N\}\). If the reshaped Toeplitz matrix \(T^g(t+M-k,k,M+N-k)\) is (strictly) \(k\)-sign consistent, then \(A\) is (strictly) \(k\)-sign consistent. In particular, the Toeplitz operator \(\Toep{g}\) is (strictly) \(k\)-sign consistent if and only if \(T^g(N-k,k,2N-k)\) is (strictly) \(k\)-sign consistent for all \(N \ge k\). \end{cor}
Unfortunately, for circulant matrices \(C^g = T^g(0,T,T)\) with \(g \in \ell_\infty(T)\), \cref{cor:toep} is often too conservative. Indeed, the reshaped Toeplitz $B := T^g(T-k,k,2T-k)$ is required to be (strictly) \(k\)-sign consistent. Due to the \(T\)-periodicity of \(g\), this implies, e.g., that
\begin{equation*}
    |B_{(:,(1:k))}| \qquad \text{and} \qquad |B_{(:,(2:k)\cup \{T+1\})}| = -|B_{(:,(1:k))}|
\end{equation*}
must share the same (strict) sign whenever \(k < T\). Hence, in the strict case, \cref{cor:toep} cannot be applied. 
\begin{example}
As shown in \cite{grussler2025discrete}, an example of a $g \in \ell_\infty(T)$ with strictly $3$-sign consistent $C^g$ is given by
    \begin{equation*}
        g_t = \frac{1}{1-0.9^T}0.9^{t-1} - \frac{2}{1-0.8^T} 0.8^{t-1} + \frac{1}{1-0.7^T}0.7^{t-1}, \; t \in (1:T).
    \end{equation*}
    In particular, if $T = 4$, then 
    \begin{equation*}
        B = T^g(1,3,5) = \begin{bmatrix}
            g_1 & g_0 & g_3 & g_2 & g_1 \\
            g_2 & g_1 & g_4 & g_3 & g_2 \\
            g_3 & g_2 & g_0 & g_4 & g_3
        \end{bmatrix}
    \end{equation*}
 with $0 < |B_{(:,(1:3))}| = -|B_{(:,(2:3)\cup \{5\})}|$, i.e., $B$ is not $3$-sign consistent. Thus, \cref{cor:toep} cannot be used to certify strict $3$-sign consistentcy of $C^g$. 
\end{example}
In the non-strict case, however, partial extensions are still possible. To see this, observe that circulant matrices cannot be \(k\)-sign consistent for even \(k\), except in trivial cases. The following example illustrates this obstruction.
\begin{example}
   The circulant matrix
    \begin{equation*}
        C^g = \begin{bmatrix}
            g_0 & g_2 & g_1 \\
            g_1 & g_0 & g_2\\
            g_2 & g_1 & g_0
        \end{bmatrix}
    \end{equation*}
    contains the $2$-minors 
    \begin{equation*}
        |C^g_{((1:2),(2:3))}| = g_2^2 - g_0g_1 = -|C^g_{(\{1,3\},(1:2))}| \qquad \text{and} \qquad |C^g_{((2:3),(1:2))}| = g_1^2 - g_0g_2 = -|C^g_{(\{1,3\},(2:3))}|.
    \end{equation*}
  $C^g$ can, thus, not be strictly $2$-sign consistent. Moreover, assuming that all $g_i \neq 0$, $2$-sign consisteny of $C^g$ is equivalent to $\frac{g_1}{g_2} = \frac{g_2}{g_0} = \frac{g_0}{g_1}$, i.e., all $g_i = \alpha$ for some $\alpha \in \mathds{C}\setminus \{0\}$. 
\end{example}
If $k$ is odd and such that $2T -k \leq T+k-1$, it follows that, up to additionally created zero $k$-minors, the $k$-minors of $T^g(T-k,k,2T-k)$ are identical to those of $T^g(T-k,k,T)$, or equivalently, to those of $C^g_{((1:k),:)}$. Therefore, by application of \cref{cor:toep}, the following equivalence holds.
\begin{cor}\label{cor:circulant}
 Let $g \in \ell_\infty(T)$ and $k \in (1:T)$ be odd such that $2k \geq T+1$. Then, the following are equivalent:
\begin{enumerate}
    \item $C^g$ is $k$-sign consistent.
    \item $C^g_{((1:k),:)}$ is $k$-sign consistent. 
\end{enumerate}
\end{cor}
Note that in the cases where \(C^g\) is strictly \(1\)-sign consistent, it has been shown in \cite{grussler2025discrete} that \cref{cor:circulant} is valid for \(k=3\), independently of \(T\).

\section{Conclusion}\label{sec:conculsion}

We have developed a framework for efficiently verifying (strict) $k$-sign consistency of Hankel and Toeplitz matrices --- meaning that all their $k$-th order minors share the same (strict) sign. While our results provide only sufficient conditions in the finite-dimensional matrix setting, they become necessary and sufficient for the corresponding infinite-dimensional operators. This generalizes the classical certificates for $k$-positivity (i.e., nonnegativity of all minors of order up to $k$) of Hankel and Toeplitz matrices \cite{grussler2020variation,karlin1968total,fallat2017total}. It, further, enables the direct use of tools such as \cite{pena_matrices_1995}, which reduce the verification of $k$-sign consistency for a matrix with $k$ rows or columns to a total positivity problem.

Our derivations rely on algebraic combinatorics -- specifically Schur polynomials, Kostka numbers, and Littlewood--Richardson coefficients \cite{stanley2023enumerative} --- to express every $k$-th order minor of a matrix as a nonnegative integer linear combination of intermediate row-consecutive $k$-th order minors. These minors arise from submatrices whose columns are formed from $k$ consecutive row indices, which need not coincide across columns. For Hankel matrices, a structural identity links the signed sums of the associated coefficients (Kostka numbers) to Littlewood--Richardson coefficients, allowing us to replace these row-consecutive minors with standard $k$-th order minors of a reshaped Hankel matrix with $k$ rows. 

Beyond their immediate implications, the connections uncovered here between algebraic combinatorics and total positivity open several promising directions for further research. These include extending analogous decompositions to other structured matrix classes, as well as exploring their consequences in system theory \cite{grussler2020variation,grussler2020balanced,margaliot2018revisiting,weiss2019generalization2} and sparse optimization \cite{marmary2025tractabledownfallbasispursuit}.

\printbibliography

\end{document}